%% file: non-holomorphic-hurwitz-congruences.tex
\documentclass[a4paper,12pt,oneside,headsepline]{scrartcl}

\input{packages}
\input{layout}

\let\cal\undefined

\input{shortcuts}

\input{shortcuts_this}

\newcommand{\headertitle}{{\normalfont%
  Congruences for Hurwitz Class Numbers
}}
\newcommand{\headerauthors}{%
  O.~Beckwith, M.~Raum, O.~K.~Richter
}
\title{%
  Non-Holomorphic\\Ramanujan-type Congruences for\\Hurwitz Class Numbers
}
\author{%
Olivia Beckwith%
\and
Martin Raum%
\thanks{The author was partially supported by Vetenskapsr\aa det Grants~2015-04139 and~2019-03551.}%
\and
Olav K. Richter
\thanks{The author was partially supported by Simons Foundation Grant~\#412655.}
}

\begin{document}

\thispagestyle{scrplain}
\begingroup
\deffootnote[1em]{1.5em}{1em}{\thefootnotemark}
\maketitle
\endgroup

{%
\noindent
{\tbf Abstract:}
In contrast to all other known Ramanujan-type congruences, we discover that Ra\-ma\-nu\-jan-type congruences for Hurwitz class numbers can be supported on non-holo\-morphic generating series. We establish a divisibility result for such non-holomorphic congruences of Hurwitz class numbers. The two keys tools in our proof are the holomorphic projection of products of theta series with a Hurwitz class number generating series and a theorem by Serre, which allows us to rule out certain congruences.
\\[.35em]
\textsf{\textbf{%
  Hurwitz class numbers%
}}%
\hspace{0.3em}{\tiny$\blacksquare$}\hspace{0.3em}%
\textsf{\textbf{%
  Ramanujan-type congruences%
}}%
\hspace{0.3em}{\tiny$\blacksquare$}\hspace{0.3em}%
\textsf{\textbf{%
  holomorphic projection%
}}%
\hspace{0.3em}{\tiny$\blacksquare$}\hspace{0.3em}%
\textsf{\textbf{%
  Che\-bo\-ta\-rev density theorem%
}}
\\[0.15em]
\noindent
\textsf{\textbf{%
  MSC Primary:
  11E41%
}}%
\hspace{0.3em}{\tiny$\blacksquare$}\hspace{0.3em}%
\textsf{\textbf{%
  MSC Secondary:
  11F33, 11F37
}}
}
\vspace{1\baselineskip}

\Needspace*{4em}
\addcontentsline{toc}{section}{Introduction}
\markright{Introduction}
\lettrine[lines=2,nindent=.2em]{\tbf T}{he} study of class numbers for imaginary quadratic fields and the related Hurwitz class numbers has a long and rich history. Their divisibility properties were first studied as early as the 1930s~\cite{scholz-1932}, but have proved highly elusive. Such divisibility properties are directly reflected in the existence of torsion elements in class groups. The Cohen-Lenstra heuristic~\cite{cohen-lenstra-1984} has been the guiding principle in the topic, providing conjectures of a statistical nature for the factorization of class numbers; see~\cite{holmin-jones-kurlberg-mcleman-petersen-2019} for an experimentally supported refinement. However, essentially nothing is known about strictly regular patterns of divisibility as opposed to statistical patterns. In this light, it is natural to study Ramanujan-type congruences, i.e., divisibility properties on arithmetic progressions.

The study of congruences of modular forms originates with the Ramanujan congruences~\cite{ramanujan-1920} for the partition function:
\begin{gather*}
  p(5 n + 4) \equiv 0 \;\pmod{5}
\tx{,}\quad
  p(7 n + 5) \equiv 0 \;\pmod{7}
\tx{,}\quad
  p(11 n + 6) \equiv 0 \;\pmod{11}
\tx{.}
\end{gather*}
It is now known that all weakly holomorphic modular forms, including the generating function for~$p(n)$, satisfy congruences~\cite{ono-2000,ahlgren-ono-2001,treneer-2006}, which arise from the theory of Galois representations associated to modular forms. Zagier~\cite{zagier-1975} showed that the Hurwitz class numbers are Fourier coefficients of a weight~$\frac{3}{2}$ mock modular form (in today's terminology), i.e., the holomorphic part of a harmonic Maass form~\cite{ono-2004,ono-2009,duke-2014}. Much less is known about congruences in this setting. However, congruences for other weight~$\frac{3}{2}$ mock modular forms have been studied by several authors~\cite{bringmann-lovejoy-2009, ono-2011, ahlgren-bringmann-lovejoy-2011,andersen-2014,andrews-passary-sellers-2017}.

Ramanujan-type congruences for the Hurwitz class numbers~$H(D)$ have not appeared in the literature. For this work, we have employed a computer search to discover many examples of such congruences for $H(D)$, which can then be confirmed with our method of Section~\ref{sec:Proof of Main Theorem} in combination with Sturm bounds for modular forms modulo a prime; see also Remark~\ref{mainrm:it:hurwitz-congruences:cross-check} following Theorem~\ref{mainthm:hurwitz-congruences}. For instance one finds that:
\begin{gather*}
  H(5^3 n + 25) \equiv 0 \;\pmod{5}
\tx{,}\quad
  H(7^3 n + 147) \equiv 0 \;\pmod{7}
\tx{,}\quad
  H(11^3 n + 242) \equiv 0 \;\pmod{11}
\tx{.}
\end{gather*}
A common theme of all three of these congruences, written in the form~$H(an + b) \equiv 0\, \pmod{\ell}$, is that~$-b$ is a square modulo~$a$, which yields a generating series for~$H(an + b)$ that is mock modular, i.e., has a non-holomorphic modular completion. This differs from the congruences for mock theta functions of weight~$\frac{1}{2}$, which are so far only known to occur when the generating function is a holomorphic modular form~\cite{andersen-2014}. Our main theorem provides the following divisibility result for such non-holomorphic congruences:
\begin{maintheorem}
\label{mainthm:hurwitz-congruences}
Fix a prime~$\ell > 3$, $a \in \ZZ_{\ge 1}$, and $b \in \ZZ$. If $-b$ is a square modulo~$a$ and
\begin{gather*}
  H(a n + b) \equiv 0 \;\pmod{\ell}
\end{gather*}
for all integers~$n$, then~$\ell \isdiv a$.
\end{maintheorem}

\begin{mainremarkenumerate}
\item
\label{mainrm:it:data}
We used the Hurwitz-Eichler relations to compute Hurwitz class numbers~$H(D)$ for~$D < 3\cdot10^9$. We did not employ any computer algebra system, but implemented our software in C/C++/Julia from scratch, relying on the FLINT library~\cite{flint-2.5.2} only for modular arithmetics. The code is available at the second author's homepage.

\item
\label{mainrm:it:hurwitz-congruences:cross-check}
All examples of Ramanujan-type congruences of Hurwitz class numbers that we discovered, including the non-holomorphic ones, can be explained by a Hecke-type factorization of Hurwitz class numbers~\eqref{eq:class-number-formula}. Note that all such non-holomorphic Ramanujan-type congruences for~$H(n)$ satisfy the conclusion of Theorem~\ref{mainthm:hurwitz-congruences}.  While it is not known whether this factorization implies all non-holomorphic Ra\-ma\-nu\-jan-type congruences, our experimental data suggests that it does. Thus, Theorem~\ref{mainthm:hurwitz-congruences} provides evidence for this speculation.

\item
In the Appendix~\ref{sec:class-number-formula} we point out that equation ~\eqref{eq:class-number-formula} also implies congruences modulo powers of~$\ell$.

\item
There are also examples of congruences~$H(a n + b) \equiv 0 \,\pmod{\ell}$ where $-b$ is \emph{not} a square modulo~$a$. However, our computational data reveals that in those cases~$\ell$ does \emph{not} divide~$a$, which is in contrast to other known Ramanujan-type congruences (see~\cite{radu-2013,andersen-2014}). The first examples for $\ell \in \{5, 7, 11\}$ are:
\begin{gather*}
  H(3^3 n + 9) \equiv 0 \;\pmod{5}
\tx{,}\quad
  H(5^3 n + 50) \equiv 0 \;\pmod{7}
\tx{,}\quad
  H(2^9 n + 192) \equiv 0 \;\pmod{11}
\tx{.}
\end{gather*}

\item
The methods in this work likely generalize to all mock modular forms of \linebreak weight~$\frac{3}{2}$. In particular, congruences of the Andrews~$\mathrm{spt}$-function~\cite{ono-2011} should be subject to requirements similar to the ones presented in Theorem~\ref{mainthm:hurwitz-congruences}. On the other hand, both the divisibility and the square class conditions that appear in the case of mock modular forms of weight~$\frac{1}{2}$, see~\cite{andersen-2014}, seem to be of a different nature and originate in the principal part of mock modular forms.
\end{mainremarkenumerate}

The method of our proof is novel: We combine a holomorphic projection argument for products of theta series and mock modular forms, which appeared first in~\cite{imamoglu-raum-richter-2014}, with a theorem by Serre~\cite{serre-1974,serre-1976} that is rooted in the Chebotarev Density Theorem. The latter was employed by Ono~\cite{ono-2000} in order to establish his celebrated results on the distribution of the partition function modulo primes. Ono applied it to establish congruences, while our proof proceeds by contradiction and uses Serre's theorem to rule out certain congruences.

\subsection*{Acknowledgments}

The authors thank Scott Ahlgren, Jeremy Lovejoy, and the \linebreak anonymous referee for valuable suggestions.

\section{Preliminaries}

A recent reference which contains most of the necessary background material for this paper is~\cite{bringmann-folsom-ono-rolen-2018}, a more classical one on the theory of modular forms is~\cite{lang-1995}.

\subsection{Modular forms}

Let~$\Ga_0(N)$, $\Ga_1(N)$, and $\Ga(N)$ be the usual congruence sub-\linebreak groups of~$\SL{2}(\ZZ)$, $\rmM_k(\Ga)$ the space of modular forms of integral or half-integral \linebreak weight~$k$ for~$\Ga \subseteq \SL{2}(\ZZ)$, and $\bbM_k(\Ga)$ the corresponding space of harmonic Maass forms (satisfying the moderate growth condition at all cusps). We also consider \linebreak quasi-modular forms of integral weight~\cite{zagier-1994,kaneko-zagier-1995}. Moreover, $\HS$ denotes the Poincar\'e upper half plane, and throughout $\tau\in\HS$, $y=\Im(\tau)$, and $e(s\tau) := \exp(2 \pi i\, s\tau)$ for $s\in\QQ$.

For a holomorphic modular form~$f(\tau)=\sum_{m\geq 0}c(f,m)e(m\tau) \in \rmM_{2-k}(\Ga(N))$ with \linebreak $k \ne 1$ and $N \in \ZZ_{\ge 1}$, we recall its non-holomorphic Eichler integral
\begin{gather}
\label{eq:def:non-holomorphic-eichler-integral}
\begin{aligned}
  f^\ast(\tau)
\;&:=\;
  -(2 i)^{k-1}
  \int_{-\ov{\tau}}^{i\infty} \frac{\ov{G(-\ov{w})}}{(w + \tau)^k}\,d\!w
\\
&\hphantom{:}=
  \frac{\ov{c(f,0)}}{1-k}\, y^{1-k}
  \,-\,
  (4 \pi)^{k-1}
  \sum_{m < 0}
  \ov{c(f, |m|)\, }|m|^{k-1}
  \Gamma(1-k,4 \pi |m| y) e(m\tau)
\tx{,}
\end{aligned}
\end{gather}
where $\Ga$ stands for the upper incomplete Gamma-function.

\subsection{Generating series of Hurwitz class numbers}

Zagier~\cite{zagier-1975} showed that the holomorphic generating series~$\sum_D H(D) e(D \tau)$ of Hurwitz class numbers admits the following modular completion:
\begin{gather}
\label{eq:zagier-eisenstein-series}
  E_{\frac{3}{2}}(\tau)
\;:=\;
  \sum_{D = 0}^\infty H(D) e(D \tau)
\,+\,
  \frac{1}{16 \pi} \theta^\ast(\tau)
\;\in\;
  \bbM_{\frac{3}{2}}(\Gamma_0(4))
\tx{,}
\end{gather}
where
\begin{gather}
\label{eq:theta-modularity}
\begin{aligned}
  \theta
&\;:=\;
  \theta_{1,0}
\,\in\,
  \rmM_{\frac{1}{2}}(\Ga_0(4))
\tx{\, with}
\\
  \theta_{a,b}(\tau)
&\;:=\;
  \hspace{-1em}
  \sum_{\substack{n \in \ZZ\\n \equiv b \,\pmod{a}}}
  e\big( \tfrac{n^2 \tau}{a} \big)
\,\in\,
  \rmM_{\frac{1}{2}}(\Ga(4 a))
\tx{,}\quad
  a \in \ZZ_{\ge 1}, b \in \ZZ
\tx{.}
\end{aligned}
\end{gather}

For $a \in \ZZ_{\ge 1}$ and $b \in \ZZ$, we define the following operator on Fourier series expansions of non-holomorphic modular forms:
\begin{gather}
  \rmU_{a,b}\,\sum_{n \in \ZZ} c(f;\,n;\,y) e(n \tau)
\;:=\;
  \sum_{\substack{n \in \ZZ\\n \equiv b \,\pmod{a}}}
  c\big( f;\,n;\,\tfrac{y}{a} \big)
  e\big( \tfrac{n \tau}{a} \big)
\tx{.}
\end{gather}
The holomorphic part of $\rmU_{a,b}\, E_{\frac{3}{2}}(\tau)$ is the generating series of Hurwitz class numbers~$H(a n + b)$ for~$n \in \ZZ$.  A Hecke-theory-like computation (see also~\cite{cohen-1975,jochnowitz-2004-preprint}) shows that 
\begin{gather}
\label{eq:eisenstein-ab-modularity}
  \rmU_{a,b}\, E_{\frac{3}{2}}
\,\in\,
  \rmM_{\frac{3}{2}}(\Ga(4 a))
\tx{.}
\end{gather}
Moreover, we have
\begin{gather}
\label{eq:Uab-theta}
  \rmU_{a,b}\,\theta
\;=\;
  \sum_{\beta^2 \equiv b \,\pmod{a}}
  \theta_{a,\beta}
\quad\tx{and}\quad
  \rmU_{a,b}\,\theta^\ast
\;=\;
  \sum_{\beta^2 \equiv -b \,\pmod{a}}
  \sqrt{a}\, \theta^\ast_{a, \beta}
\tx{.}
\end{gather}
In particular, if~$-b$ is not a square modulo~$a$, then $\rmU_{a,b}\,E_{3 \slash 2}$ is a holomorphic modular form.

\subsection{A theorem by Serre}

The following theorem by Serre and its corollary allow us to disprove that a given generating series is a quasi-modular form modulo a prime. Recall that a rational number is called $\ell$-integral for a prime~$\ell$, if its denominator is not divisible by~$\ell$.

\begin{theorem}[{Deligne-Serre~\cite{deligne-serre-1974} and Serre~\cite{serre-1974,serre-1976}}]
\label{thm:serre}
Fix an odd prime~$\ell$ and \linebreak $k, N \in \ZZ_{\ge 1}$. Then there are infinitely many primes~$p \equiv 1 \,\pmod{\ell N}$ such that for every~$f \in \rmM_k(\Ga_1(N))$ with $\ell$-integral Fourier coefficients, we have
\begin{gather}
\label{eq:thm:serre}
  c(f;\, n p^r)
\;\equiv\;
  (r + 1) c(f;\, n)
  \;\pmod{\ell}
\end{gather}
for all $n \in \ZZ$ co-prime to~$p$ and all~$r \in \ZZ_{\ge 0}$.
\end{theorem}
\begin{proof}
The proof of Lemma~9.6 of~\cite{deligne-serre-1974}, which is stated in the special case of weight~$1$, extends verbatim.
\end{proof}

\begin{corollary}
\label{cor:serre}
Fix a prime~$\ell > 3$ and positive integers~$k$ and~$N$. Then there are infinitely many primes~$p \equiv 1 \,\pmod{\ell N}$ such that for every quasi-modular form~$f$ of weight~$k$ for~$\Ga_1(N)$ with $\ell$-integral Fourier coefficients, we have the congruence~\eqref{eq:thm:serre}.
\end{corollary}
\begin{proof}
Recall the weight~$2$ quasi-modular form
\begin{gather*}
  E_2(\tau)
\;=\;
  1 - 24 \sum_{n = 1}^\infty \sigma_1(n) e(n\tau)
\tx{,}\quad
  \sigma_1(n)
\;:=\;
  \sum_{\substack{d \isdiv n\\d > 0}} d
\tx{.}
\end{gather*}
Quasi-modular forms are polynomials in~$E_2$ whose coefficients are modular forms. More specifically, a quasi-modular form of weight~$k$ for~$\Ga_1(N)$ can be written as
\begin{gather*}
  \sum_{n = 0}^d
  E_2^n f_n
\tx{,}\quad
  f_n
  \in
  \rmM_{k-2n}(\Ga_1(N))
\tx{.}
\end{gather*}
Recall also that the weight~$\ell - 1$ and level~$1$ Eisenstein series~$E_{\ell - 1}$ is a modular form (here we use that $\ell > 3$) that is congruent to~$1$ modulo~$\ell$, and that the weight~$\ell + 1$ and level~$1$ Eisenstein series~$E_{\ell + 1}$ is a modular form that is congruent to~$E_2$ modulo~$\ell$. Thus,
\begin{gather*}
  \sum_{n = 0}^d
  E_2^n f_n
\;\equiv\;
  E_{\ell - 1}^d
  \sum_{n = 0}^d
  E_2^n f_n
\;\equiv\;
  \sum_{n = 0}^d
  (E_2 E_{\ell - 1})^n
  E_{\ell - 1}^{d -n}
  f_n
\;\equiv\;
  \sum_{n = 0}^d
  E_{\ell + 1}^n
  E_{\ell - 1}^{d -n}
  f_n
  \;\pmod{\ell}
\tx{,}
\end{gather*}
which allows us to apply Theorem~\ref{thm:serre} to the modular form of weight~$k + d(\ell-1)$ on the right hand side.
\end{proof}

\subsection{Holomorphic projection}

We now revisit holomorphic projection, which allows one to map continuous functions with certain growth and modular behavior to holomorphic modular forms (for example, see~\cite{sturm-1980,gross-zagier-1986}).  It is convenient for us to refer to~\cite{imamoglu-raum-richter-2014} as a reference, since it provides a variant that does not project to cusp forms. Fix a weight~$k \in \ZZ$, $k \ge 2$ and a level~$N \in \ZZ_{\ge 1}$.

Consider an~$N$-periodic continuous function $f :\,\HS \ra \CC$ with Fourier series expansion
\begin{gather*}
  f(\tau)
\;=\;
  \sum_{n \in \frac{1}{N}\ZZ}
  c(f;\,n;\,y) e(\tau n)
\end{gather*}
subject to the conditions:
\begin{enumerateroman*}
\item
For some $a > 0$ and all~$\ga \in \SL{2}(\ZZ)$, there are coefficients $\td{c}(f|_k\,\ga;\,0) \in \CC$, such that $(f |_k\,\ga)(\tau) = \td{c}(f|_k\,\ga;\,0) + \cO(y^{-a})$ as $y \ra \infty$;

\item
For all $n \in \frac{1}{N}\ZZ_{> 0}$, we have $c(f;\, n;\, y) = \cO(y^{2-k})$ as $y \ra 0$.
\end{enumerateroman*}
The holomorphic projection operator of weight~$k$ is defined by
\begin{gather}
\label{eq:def:holomorphic-projection}
\begin{aligned}
  \pi^\hol_k(f)
\;&:=\;
  \td{c}(f;\,0)
  \,+\,
  \sum_{n \in \frac{1}{N}\ZZ_{>0}}
  c\big(\pi^\hol_k(f);\, n\big)\, e(n\tau)
\quad\tx{with}
\\
  c\big(\pi^\hol_k(f);\, n\big)
\;&:=\;
  \frac{(4 \pi n)^{k-1}}{\Gamma(k-1)}\,
  \lim_{s \ra 0}\,
  \int_0^\infty
  c(f;\, n;\, y) \exp(-4 \pi n y) y^{s+k-2}\,\rmd\!y
\tx{.}
\end{aligned}
\end{gather}
\begin{remark}
The formulation in~\cite{imamoglu-raum-richter-2014} is missing the limit~$s \ra 0$, which is required to ensure convergence in all situations. Nevertheless, the arguments in~\cite{imamoglu-raum-richter-2014} are still valid without any restrictions after adding that missing regularization.
\end{remark}

Proposition~4 and Theorem~5 of~\cite{imamoglu-raum-richter-2014} provide some key properties of the holomorphic projections operator (for vector-valued modular forms):  If $f$ is holomorphic, then $\pi^\hol_k(f) = f$.  Furthermore, if $f$ transforms like a modular form of weight~$2$ for the group~$\Ga_1(N)$, then~$\pi^\hol_2(f)$ is a quasi-modular form of weight~$2$ for $\Ga_1(N)$.

\section{Proof of Theorem~\ref{mainthm:hurwitz-congruences}}
\label{sec:Proof of Main Theorem}

We will apply holomorphic projection to products of holomorphic modular forms of weight~$\frac{1}{2}$ and harmonic Maass forms of weight~$\frac{3}{2}$. These products when inserted into~\eqref{eq:def:holomorphic-projection} lead to the integrals evaluated in the next two lemmas.

\begin{lemma}
\label{la:holomorphic-projection-constant-contribution}
Given~$n \in \QQ_{> 0}$, we have
\begin{gather*}
  \pi^\hol_2\Big(
  y^{-\frac{1}{2}}\, e(n \tau)
  \Big)
\;=\;
  2 \pi \sqrt{n}\, e(n \tau)
\tx{.}
\end{gather*}
\end{lemma}
\begin{proof}
We suppress the limit~$s \ra 0$ from our calculation, since the integral is convergent at~$s = 0$. We then have to evaluate
\begin{gather*}
  \frac{4 \pi n}{\Gamma(1)}
  \int_0^\infty
  y^{-\frac{1}{2}}\, \exp(- 4 \pi n y) \,\rmd\!y
=
  \frac{4 \pi n}{\Gamma(1)}
  (4 \pi n)^{-\frac{1}{2}} \Gamma(\tfrac{1}{2})
=
  2 \pi \sqrt{n}
\tx{.}
\end{gather*}
\end{proof}

\begin{lemma}
\label{la:holomorphic-projection-cuspidal-contribution}
Given rational numbers $n, \td{n} \in \QQ$ satisfying $\td{n} < 0$ and $n + \td{n} > 0$, we have
\begin{gather}
  \pi^\hol_2\Big(
    \Gamma\big(-\tfrac{1}{2},\, 4 \pi |\td{n}| y\big)
    e\big( (n + \td{n}) \tau \big)
  \Big)
=
  \frac{2 \sqrt{\pi} (n + \td{n})}{\sqrt{|\td{n}|} \,(\sqrt{n} + \sqrt{|\td{n}|})}\,
    e\big( (n + \td{n}) \tau \big)
\tx{.}
\end{gather}
\end{lemma}
\begin{proof}
Again, we suppress the limit~$s \ra 0$ from our calculation, since the integral is convergent at~$s = 0$. We have to evaluate
\begin{gather*}
  \int_0^\infty
  \Gamma\big(-\tfrac{1}{2},\, 4 \pi |\td{n}| y\big)
  \exp(-4 \pi (n + \td{n}) y) \,\rmd\!y
\tx{.}
\end{gather*}
We apply case~1 of (6.455) on p.~657 of~\cite{gradshteyn-ryzhik-2007} with~$\alpha \leadsto 4 \pi |\td{n}|$, $\beta \leadsto 4 \pi (n + \td{n})$, $\mu \leadsto 1$, and $\nu \leadsto -\frac{1}{2}$. The assumptions $\Re(\alpha + \beta) = 4 \pi ((n + \td{n}) + |\td{n}|) > 0$ (since $n + \td{n} > 0$), $\Re(\mu) = 1 > 0$, and $\Re(\mu + \nu) = \frac{1}{2} > 0$ are satisfied. As a result, we obtain
\begin{gather*}
  \frac{(4 \pi |\td{n}|)^{-\frac{1}{2}}\, \Gamma(\frac{1}{2})}
       {1 \cdot (4 \pi ((n + \td{n}) + |\td{n}|))^{\frac{1}{2}}}\,
  {}_2\rmF_1\Big(1,\tfrac{1}{2};\,2;\, \frac{n + \td{n}}{n + \td{n} + |\td{n}|} \Big)
=
  \frac{1}{4 \sqrt{\pi n |\td{n}|}}\,
  {}_2\rmF_1\Big(1,\tfrac{1}{2};\,2;\, \frac{n + \td{n}}{n} \Big)
\tx{.}
\end{gather*}
To evaluate the hypergeometric function, we employ 15.4.17 of~\cite{nist-dlmf-1-0-27} with~$a \leadsto \frac{1}{2}$. It allows us to simplify the previous expression to
\begin{gather*}
  \frac{1}{4 \sqrt{\pi n |\td{n}|}}\,
  \Big(\frac{1}{2} + \frac{1}{2} \sqrt{1 - \frac{n + \td{n}}{n}} \Big)^{-1}
=
  \frac{1}{2 \sqrt{\pi |\td{n}|} \,(\sqrt{n} + \sqrt{|\td{n}|})}
\tx{.}
\end{gather*}
We combine this with the leading factor $4 \pi (n + \td{n})$ in the defining equation~\eqref{eq:def:holomorphic-projection} of the holomorphic projection to finish the proof.
\end{proof}

The next three results compute the $n$\thdash\ coefficient of~$\pi_2^\hol \big( (\rmU_{a,b} E_{\frac{3}{2}}) \cdot ( \theta_{a,\beta} + \theta_{a,-\beta}) \big)$ for certain~$n$. The first elementary lemma establishes the existence of a subprogression of~$a \ZZ + b = \{ an + b : n \in \ZZ \}$ satisfying arithmetic conditions that will be useful in the proof of Theorem A. To state it, we let~$\ordp(a)$ be the maximal exponent for powers of a prime~$p$ dividing an nonzero integer~$a$.
\begin{lemma}
\label{la:subprogression}
Let~$\td{a} \in \ZZ_{\ge 1}$ and~$\td{b} \in \ZZ$ be such that~$-\td{b}$ is a square modulo~$\td{a}$. Denote by~$P$ the set of prime divisors of~$\td{a}$. Then there exist~$a \in \ZZ_{\ge 1}$ and $b \in \ZZ$ such that 

\begin{enumerateroman}
\item
\label{it:la:subprogression:first-condition}
\label{it:la:subprogression:primes-in-a}
We have~$\td{a} \isdiv a$, and if~$p$ is a prime divisor of~$a$ then~$p \in P$.

\item
\label{it:la:subprogression:-b-is-square}
The integer~$-b$ is a square modulo~$a$. 

\item
\label{it:la:subprogression:valuations}
We have 
\begin{gather*}
  h(p)
\;>\;
  \begin{cases}
    2 k(p)\tx{,} & \tx{if } p \neq 2 \tx{;} \\
    2 k(p) + 4\tx{,} & \tx{if } p = 2 \tx{;}
  \end{cases}
\end{gather*}
where $h(p) = \ordp(a)$ and $k(p) = \ordp(b)$.

\item
\label{it:la:subprogression:last-condition}
\label{it:la:subprogression:-4b-not-in-prime-products}
For~$p \in P$ and for all disjoint sets~$P_1, P_2$ with $P_1 \cup P_2 = P \backslash \{p \}$, we have
\begin{gather}
\label{eq:it:la:subprogression:-4b-not-in-prime-products}
  -4b
\;\not\equiv\;
  \left\{
  \begin{alignedat}{3}
  &
    p^{k(p)}
    \prod_{q \in P_1} q^{k(q)}
    \prod_{q \in P_2} q^{2 h(q)}
  &&
    \;\pmod{ p^{h(p)}} 
  \tx{,}\qquad
  &&
    \tx{if } p \neq 2
  \tx{,}\\
  &
    p^{k(p) + 2}
    \prod_{q \in P_1} q^{k(q)}
    \prod_{q \in P_2} q^{h(q)}
  &&
    \;\pmod{ p^{h(p)}}
  \tx{,}\qquad
  &&
    \tx{if } p = 2
  \tx{.}
  \end{alignedat}
  \right.
\end{gather}
\end{enumerateroman}
\end{lemma}

\begin{proof}
In order to produce integers~$a$ and~$b$ that satisfy Conditions~\ref{it:la:subprogression:first-condition}--\ref{it:la:subprogression:last-condition}, we initially set $a = \td{a}$ and $b = \td{b}$ and then repeatedly replace~$a$ and~$b$ by~$a a'$ and~$b + a b'$ for integers~$a', b'$ in such a way that successively more of these conditions are met. In accordance with Condition~\ref{it:la:subprogression:primes-in-a}, the prime divisors of~$a'$ must be elements of~$P$.

Recall that~$h(p) = \ordp(a)$ and~$k(p) = \ordp(b)$. For each prime~$p \in P$, we define $a_p := a \slash p^{h(p)}$ and similarly define~$b_p := b \slash p^{k(p)}$. We start by making the substitutions~$a \leadsto a a'$ and~$b \leadsto b + a b'$ several times, in such a way that~$-(b + a b')$ is still a square modulo~$a a'$ (hence Condition~\ref{it:la:subprogression:-b-is-square} remains true) and that Condition~\ref{it:la:subprogression:valuations} is satisfied.

First choose $a' = 1$, and pick a suitable~$b'$ such that $h(p) \ge k(p)$ holds for each~$p \in P$ after replacing $b$ by~$b + a b'$. Next let~$P_0 := \{ p \in P \,:\, h(p) = k(p) \}$. For each~$p \in P_0$, let~$e(p) = 1$ if $h(p)$ is even, and let $e(p) = 2$ if~$h(p)$ is odd. Consider $a' = \prod_{p \in P_0} p^{e(p)}$. Let~$b'$ be any integer satisfying~$\gcd(b', \prod_{p \in P \backslash P_0} p ) = 1$ and
\begin{gather*}
  b'
\;\equiv\;
  \left\{
  \begin{alignedat}{3}
  &
    -(1 + b_p) a_p^{-1}
  &&
    \;\pmod{p}
  \tx{,}\qquad
  &&
    \tx{if $p \in P_0$ and $2 \isdiv h_1(p)$;}
  \\
  &
    -(p + b_p) a_p^{-1}
  &&
    \;\pmod{p^2}
  \tx{,}
  &&
    \tx{if $p \in P_0$ and $2 \nisdiv h_1(p)$.}
  \end{alignedat}
  \right.
\end{gather*}
Note that the minus signs of the terms~$-(1 + b_p)$ and~$-(p + b_p)$ are required to ensure that Condition~\ref{it:la:subprogression:-b-is-square} holds after replacing~$a$ and~$b$ by~$a a'$ and~$b + a b'$. After this substitution we may assume that $h(p) > k(p)$ for all $p \in P$.

We take this idea a step further when $p = 2$. If $h(2) - k(2) = 1$, then we choose $a' = 2$, and $b'$ is~$0$ or~$1$ depending on whether $b_p$ is~$3$ or~$1$ modulo~$4$. We find that we may assume~$h(2) - k(2) \ge 2$. Similarly, if we assume that $h(2) - k(2) = 2$, then we can choose $a' = 2$, and $b'$ is~$0$ or~$1$ depending on whether $b_p$ is~$7$ or~$3$ modulo~$8$. Thus, we can assume that $h(2) - k(2) \ge 3$.

After making the above assumptions on~$h(p)$ and~$k(p)$, we conclude that~$-b$ is also a square modulo~$a a'$ for any~$a'$ satisfying Condition~\ref{it:la:subprogression:primes-in-a}. In particular, after choosing an appropriate~$a'$, we may assume that~$h(p)$ is as large as we wish. Thus, we can assume that Condition~\ref{it:la:subprogression:valuations} holds.

It remains to make one more substitution~$a \leadsto a a'$ and~$b \leadsto b + a b'$ in order to ensure Condition~\ref{it:la:subprogression:-4b-not-in-prime-products}. Observe that the right hand side of~\eqref{eq:it:la:subprogression:-4b-not-in-prime-products} can take at most~$2^{|P|-1}$ different values, corresponding to the $2^{|P|-1}$ different choices of~$P_1$ and~$P_2$. Given~$a'$, we let~$h'(p) := \ordp(a')$ for~$p \in P$. Any choice of~$b'$ preserves Conditions~\ref{it:la:subprogression:primes-in-a}, \ref{it:la:subprogression:-b-is-square}, and~\ref{it:la:subprogression:valuations}, and this yields~$p^{h'(p)}$ different values for the left hand side of~\eqref{eq:it:la:subprogression:-4b-not-in-prime-products} after replacing~$b$ by~$b + a b'$. Finally, choose~$a'$ in such a way that~$p^{h'(p)} > 2^{|P|-1}$ for every~$p \in P$, and then pick a suitable~$b'$ to validate Condition~\ref{it:la:subprogression:-4b-not-in-prime-products}.
\end{proof}

The next result will be useful in the computation of the Fourier series coefficients of~$\pi_2^\hol \big( (\rmU_{a,b} E_{\frac{3}{2}}) \cdot ( \theta_{a,\beta} + \theta_{a,-\beta}) \big)$. We continue using~$P$ to denote the set of prime divisors of~$a$, and we assume the definition of~$h(p)$ and~$k(p)$ from Lemma~\ref{la:subprogression}. We also let $\beta$ be such that $\beta^2 \equiv -b \pmod{a}$. Thus, we have~$a = \prod_{p \isdiv a} p^{h(p)}$, and~$\beta = \beta' \prod_{p \isdiv a} p^{k(p) \slash 2}$, where~$\beta'$ is co-prime to~$a$. 

\begin{lemma}
\label{la:nonarchimedian-contribution}
Assume that~$a \in \ZZ_{\ge 1}$ and~$b \in \ZZ$ satisfy Conditions~\ref{it:la:subprogression:first-condition}--\ref{it:la:subprogression:last-condition} in Lemma~\ref{la:subprogression}, and let~$\ell$ be a prime that does not divide~$a$.

For each prime~$p \in P$, let~$q(p)$ be a prime that is congruent~$1$ modulo~$p^{h(p)}$, congruent to~$2 \beta \slash p^{k(p) \slash 2}$ modulo~$p^{h(p) - k(p) \slash 2}$, and congruent modulo~$\ell$ to a unit such that
\begin{alignat*}{3}
&
  p^{h(p)} + q(p) p^{k(p)}
&&\;\equiv\; 1
  \;\pmod{\ell}
\tx{,}\quad
&&
  \tx{if~$p \ne 2$, or}
\\
&
  p^{h(p)} + q(p) p^{k(p) + 1}
&&\;\equiv\;
  1
  \;\pmod{\ell}
\tx{,}\quad
&&
  \tx{if~$p = 2$.}
\end{alignat*}
Moreover, let 
\begin{gather*}
  n_a
\;:=\;
  \left\{
  \begin{alignedat}{2}
  &
    \hphantom{8{}}
    \prod_{p \in P} q(p) p^{k(p)}
  \tx{,}\quad
  &&
    \tx{if\/ $2 \nisdiv a$;}
  \\
  &
    8
    \prod_{p \in P} q(p) p^{k(p)}
  \tx{,}\quad
  &&
    \tx{if\/ $2 \isdiv a$.}
  \end{alignedat}
  \right.
\end{gather*}

Then we have
\begin{gather*}
  \sum_{\substack{\td\beta^2 \equiv -b \,\pmod{a}\\\epsilon = \pm 1}}
  \sum_{\substack{a n_a = d_1 d_2\\d_1, d_2 > 0\\d_1 \equiv \epsilon\beta + \td\beta \,\pmod{a}\\d_2 \equiv \epsilon\beta - \td\beta \,\pmod{a}}}
  \big( d_1 + d_2 \big)
\;\equiv\;
  2
  \;\pmod{\ell}
\tx{.}
\end{gather*}
\end{lemma}

\begin{proof}
Because the symmetry~$\td\beta \mto - \td\beta$ swaps~$d_1$ and~$d_2$, we have
\begin{gather}
\label{eq:la:nonarchimedian-contribution:negation-symmetry}
  \sum_{\substack{\td\beta^2 \equiv -b \,\pmod{a}\\\epsilon = \pm 1}}
  \hspace{-0.4em}
  \sum_{\substack{a n_a = d_1 d_2\\d_1, d_2 > 0\\d_1 \equiv \epsilon\beta + \td\beta \,\pmod{a}\\d_2 \equiv \epsilon\beta - \td\beta \,\pmod{a}}}
  \hspace{-1em}
  \big( d_1 + d_2 \big)
\;=\;
  2 \cdot
  \sum_{\substack{\td\beta^2 \equiv -b \,\pmod{a}\\\epsilon = \pm 1}}
  \hspace{-0.4em}
  \sum_{\substack{a n_a = d_1 d_2\\d_1, d_2 > 0\\d_1 \equiv \epsilon\beta + \td\beta \,\pmod{a}\\d_2 \equiv \epsilon\beta - \td\beta \,\pmod{a}}}
  \hspace{-1em}
  d_1
\tx{.}
\end{gather}

If~$p \ne 2$, the square roots~$\td\beta$ of~$-b \,\pmod{p^{h(p)}}$ are of the form~$\pm \beta + p^{h(p) - k(p)} \td\beta'$ for some~$\td\beta' \in \ZZ \slash p^{k(p)} \ZZ$. In the sum, we therefore have the conditions
\begin{gather*}
  d_1
\;\equiv\;
  \epsilon \beta \pm \beta + p^{h(p) - k(p)} \td\beta'
  \;\pmod{p^{h(p)}}
\tx{,}\quad
  d_2
\;\equiv\;
  \epsilon \beta \mp \beta - p^{h(p) - k(p)} \td\beta'
  \;\pmod{p^{h(p)}}
\tx{.}
\end{gather*}
Using~$h(p) > 2 k(p)$, asserted by Condition~\ref{it:la:subprogression:valuations} of Lemma~\ref{la:subprogression}, and~\eqref{eq:it:la:subprogression:-4b-not-in-prime-products}, we find that the only possibility that $a n_a = d_1 d_2$ satisfies these congruences is if we have~$\epsilon = +1$, $\td\beta' \equiv 0 \,\pmod{p^{h(p)}}$, and
\begin{gather}
\label{eq:la:nonarchimedian-contribution:d12-divisors}
  p^{h(p)} \isdiv d_1
\tx{,}\;
  q(p) p^{k(p)} \isdiv d_2
\quad\tx{or}\quad
  q(p) p^{k(p)} \isdiv d_1
\tx{,}\;
  p^{h(p)} \isdiv d_2
\tx{.}
\end{gather}

The case of~$p = 2$ allows for more square roots of~$-b \,\pmod{2^{h(p)}}$. Specifically, \linebreak $\td\beta = \{1, c_3, c_5, c_7\} \beta + p^{h(p)-k(p)} \td\beta'$, where by the set notation we indicate one of the factors occurs and $c_3 \equiv 3 \,\pmod{8}$, $c_5 \equiv 5 \,\pmod{8}$, $c_7 \equiv 7 \,\pmod{8}$ are roots of~$1$ modulo~$2^{h(2)}$. We now have the conditions
\begin{align*}
  d_1 &\;\equiv\; \{2, 1+c_3, 1+c_5, 1+c_7\} \beta + 2^{h(2)-k(2)} \td\beta'
  \;\pmod{2^{h(2)}}
\tx{,}
\\
  d_2 &\;\equiv\; \{0, 1-c_3, 1-c_5, 1-c_7\} \beta - 2^{h(2)-k(2)} \td\beta'
  \;\pmod{2^{h(2)}}
\tx{.}
\end{align*}
We use that $k(p) + 3 < h(p) - k(p)$, asserted by Condition~\ref{it:la:subprogression:valuations} of Lemma~\ref{la:subprogression}, and~\eqref{eq:it:la:subprogression:-4b-not-in-prime-products} in order to see that the only possibilities are~$\epsilon = +1$, $\td\beta' \equiv 0 \,\pmod{2^{h(2)}}$, and the divisibilities in~\eqref{eq:la:nonarchimedian-contribution:d12-divisors} with~$p = 2$.

Observe that modulo~$\ell$, the sum actually factors for our choice of~$n_a$. Specifically, we have
\begin{multline}
\label{eq:theta-theta-holomorphic-projection-coefficient-nonarchimedean-contribution}
  \sum_{\td\beta^2 \equiv -b \,\pmod{a}}
  \sum_{\substack{a n_a = d_1 d_2\\d_1, d_2 > 0\\d_1 \equiv \beta + \td\beta \,\pmod{a}\\d_2 \equiv \beta - \td\beta \,\pmod{a}}}
  d_1
\\
\;=\;
  \prod_{\substack{p \isdiv a\\p \ne 2}}
  \big( p^{h(p)} + q(p) p^{k(p)} \big)
  \,\cdot\,
  \left\{
  \begin{alignedat}{2}
  &
    1
  \tx{,}\quad
  &&
    \tx{if $2 \nisdiv a$;}
  \\
  &
    2^{h(2)} + q(2) 2^{k(2)+1}
  \tx{,}\quad
  &&
    \tx{if $2 \isdiv a$.}
  \end{alignedat}
  \right.
\end{multline}
Our choice of~$q(p)$ ensures that this is congruent to~$1$ modulo~$\ell$. The additional factor~$2$ in~\eqref{eq:la:nonarchimedian-contribution:negation-symmetry} yields the desired result.
\end{proof}

Next we compute certain coefficients of~$\pi^\hol_2\big((\rmU_{a,b}\,E_{\frac{3}{2}}) \cdot  (\theta_{a,\beta} + \theta_{a,-\beta}) \big)$. We assume the notation above for $n_a$, $k(p)$, $h(p)$, and $\beta$. 

\begin{proposition}
\label{prop:coefficientformula}
Assume that~$H(an+b) \equiv 0 \pmod{\ell}$ for all $n$. Furthermore, assume that $\beta^2 \equiv -b \,\pmod{a}$, and that~$a$ and~$b$ satisfy Conditions~\ref{it:la:subprogression:first-condition}--\ref{it:la:subprogression:last-condition} of Lemma~\ref{la:subprogression}.

Then~$\pi^\hol_2\big((  \rmU_{a,b}\,E_{\frac{3}{2}}) \cdot  (\theta_{a,\beta} + \theta_{a,-\beta}) \big)$ is a quasi-modular form for~$\Ga_1(4 a)$ and
\begin{gather*}
  \pi^\hol_2\big((
  \rmU_{a,b}\,E_{\frac{3}{2}})
  \cdot
  (\theta_{a,\beta} + \theta_{a,-\beta})
  \big)
\;=\;
  \sum_{n = 0}^{\infty} c(n) e(n \tau)
\tx{,}
\end{gather*}
where
\begin{gather*}
  c(n_a pq) \;\equiv\; -2 (1 + q)
  \;\pmod{\ell}
\quad\tx{and}\quad
  c(n_a p^2 q) \;\equiv\; -2 (1 + p + q)
  \;\pmod{\ell}
\end{gather*}
for any primes~$p,q \equiv 1 \,\pmod{\ell}$ such that 
\begin{gather}
\label{eq:eisenstein-theta-holomorphic-projection-coefficient-prime-condition}
  p^2 > (a n_a) q > (a n_a)^2 p > (a n_a)^3
\tx{.}
\end{gather}
\end{proposition}

\begin{remark}
\label{rm:prop:coefficientformula}
In light of Lemma~\ref{la:subprogression}, the assumptions on $a$ and $b$ are no restriction. For any $a,b$ with $-b$ a square modulo $a$, we can find a subprogression of $a \ZZ + b$ satisfying these conditions. Furthermore, for any sufficiently large $p \equiv 1 \,\pmod{a \ell}$, there is a prime $q \equiv 1 \,\pmod{\ell}$ satisfying the Conditions in~\eqref{eq:eisenstein-theta-holomorphic-projection-coefficient-prime-condition}---this follows from the Prime Number Theorem for arithmetic progressions. 
\end{remark}

\begin{proof}
As we assume that~$H(a n + b) \equiv 0 \,\pmod{\ell}$ for all~$n \in \ZZ$, we have
\begin{align}
\nonumber
&
  \pi^\hol_2\big(
  (\rmU_{a,b}\,E_{\frac{3}{2}})
  \cdot
  (\theta_{a,\beta} + \theta_{a,-\beta})
  \big)
\\
\nonumber
\;=\;&
  \big( \rmU_{a,b}\sum_{D = 0}^\infty H(D) e(D \tau) \big)
  \cdot
  (\theta_{a,\beta} + \theta_{a,-\beta})
  \,+\,
  \frac{1}{16 \pi}
  \pi^\hol_2\big(
  (\rmU_{a,b}\,\theta^\ast)
  \cdot
  (\theta_{a,\beta} + \theta_{a,-\beta})
  \big)
\\
\label{eq:eisenstein-theta-holomorphic-projection}
\;\equiv\;&
  \frac{1}{16 \pi}
  \pi^\hol_2\big(
  (\rmU_{a,b}\,\theta^\ast)
  \cdot
  (\theta_{a,\beta} + \theta_{a,-\beta})
  \big)
  \;\pmod{\ell}
\tx{.}
\end{align}
Now~\eqref{eq:zagier-eisenstein-series} and~\eqref{eq:Uab-theta} lead us to the study of~$\sqrt{a}\, \theta^\ast_{a,\beta}$, which arise from~$\rmU_{a,b}\,\theta^\ast$. Write $\delta_\bullet$ for the Kronecker~$\delta$-function. For general $\beta, \td\beta \in \ZZ$, we use~\eqref{eq:def:non-holomorphic-eichler-integral} to compute that
\begin{align*}
  \sqrt{a}\, \theta^\ast_{a,\td\beta}(\tau) \cdot \theta_{a,\beta}(\tau)
\;=\;
&
  -
  2 \sqrt{a}
  \delta_{\td\beta \equiv 0 \,\pmod{a}}\,
  \sum_{m \equiv \beta \,\pmod{a}}
  y^{-\frac{1}{2}}
  e\big( \frac{m^2 \tau}{a} \big)
\\
&
  -
  2 \sqrt{\pi}
  \sum_{\substack{m \equiv \beta \,\pmod{a}\\\td{m} \equiv \td{\beta} \,\pmod{a}\\\td{m} \ne 0}}
  |\td{m}|
  \Gamma\big(-\tfrac{1}{2},\, 4 \pi \frac{\td{m}^2}{a} y \big)
  e\big( \frac{(m^2 - \td{m}^2) \tau}{a} \big)
\tx{.} 
\end{align*}
We apply Lemmas~\eqref{la:holomorphic-projection-constant-contribution} and~\eqref{la:holomorphic-projection-cuspidal-contribution} to find that $ \pi^\hol_2\big( \sqrt{a}\, \theta^\ast_{a,\td\beta}(\tau) \cdot \theta_{a,\beta}(\tau) \big)$ is equal to
\begin{gather*}
  -4 \pi
  \Big(
  \delta_{\td\beta \equiv 0 \,\pmod{a}}\,
  \sum_{\substack{m \equiv \beta \,\pmod{a}\\m \ne 0}}
  |m|
  e\big( \frac{m^2 \tau}{a} \big)
  \,+\,
  \sum_{\substack{m \equiv \beta \,\pmod{a}\\\td{m} \equiv \td{\beta} \,\pmod{a}\\\td{m} \ne 0}}
  \frac{m^2 - \td{m}^2}{|m| + |\td{m}|}
  e\big( \frac{(m^2 - \td{m}^2) \tau}{a} \big)
  \Big)
\tx{.}
\end{gather*}

Summarizing, we find that 
\begin{multline}
\label{eq:eisenstein-theta-holomorphic-projection-coefficient}
  c\Big( \pi^\hol_2\big(
  \sqrt{a}\, \theta^\ast_{a,\td\beta}(\tau) \cdot \theta_{a,\beta}(\tau)
  \big);\,
  n
  \Big)
\\
=\;
  -4 \pi
  \Big(
  \delta_{\td\beta \equiv 0 \,\pmod{a}}\,
  \sum_{\substack{m \equiv \beta \,\pmod{a}\\m \ne 0\\a n = m^2}}
  |m|
  \,+\,
  a n\hspace{-1em}
  \sum_{\substack{m \equiv \beta \,\pmod{a}\\\td{m} \equiv \td{\beta} \,\pmod{a}\\\td{m} \ne 0\\a n = m^2 - \td{m}^2}}
  \frac{1}{|m| + |\td{m}|}
  \Big)
\tx{.}
\end{multline}
We can drop the first contribution, since $-b \not\equiv 0 \,\pmod{a}$ and hence~$\td\beta \not\equiv 0 \,\pmod{a}$. It remains to analyze the second term on the right hand side.

When writing $an = m^2 - \td{m}^2 = (m + \td{m})(m - \td{m})$ we recognize that the summation runs over factorizations of~$an$. Assume that $an = d_1 d_2$ is a factorization corresponding to~$(m + \td{m})(m - \td{m})$, then $m = (d_1 + d_2) \slash 2$ and $\td{m} = (d_1 - d_2) \slash 2$. Since $an > 0$, we conclude that $d_1$ and~$d_2$ have the same sign. We treat only the positive case; the negative case yields the same sum and hence contributes an additional factor of~$2$ in~\eqref{eq:theta-theta-holomorphic-projection-coefficient-d12-factorizaton}.

Since~$d_1, d_2 > 0$, we have~$m > 0$. If we assume that~$a n$ is not a square, then~$\td{m} \ne 0$ and the sign of~$\td{m}$ is positive if~$d_1 > d_2$ and negative if~$d_2 > d_1$. As a result we find that $|m| + |\td{m}|$ equals the larger factor in~$d_1 d_2$. Summarizing, we have
\begin{align}
\nonumber
  a n\hspace{-1em}
  \sum_{\substack{m \equiv \beta \,\pmod{a}\\\td{m} \equiv \td{\beta} \,\pmod{a}\\\td{m} \ne 0\\a n = m^2 - \td{m}^2}}
  \frac{1}{|m| + |\td{m}|}
&\nonumber
\;=\;
  2
  \sum_{\substack{a n = d_1 d_2\\d_1, d_2 > 0\\d_1 \equiv \beta + \td\beta \,\pmod{a}\\d_2 \equiv \beta - \td\beta \,\pmod{a}}}
  \frac{d_1 d_2}{d_1 \delta_{d_1 > d_2} + d_2 \delta_{d_2 > d_1}}
\\
\label{eq:theta-theta-holomorphic-projection-coefficient-d12-factorizaton}
&\;=\;
  2
  \sum_{\substack{a n = d_1 d_2\\d_1, d_2 > 0\\d_1 \equiv \beta + \td\beta \,\pmod{a}\\d_2 \equiv \beta - \td\beta \,\pmod{a}}}
  \big( d_1 \delta_{d_1 < d_2} + d_2 \delta_{d_2 < d_1} \big)
\tx{.}
\end{align}
We next want to separate the archimedean and nonarchimedean conditions on the right hand side.

We write $n = n_a n'$, where $n'$ is either $pq$ or $p^2 q$ and where $p$ and $q$ are as in the statement of the proposition. Then for any factorization $n' = d'_1 d'_2$, we have \linebreak $d'_1 \slash d'_2 > a n_a$ or $d'_2 \slash d'_1 > a n_a$. This assumption ensures that in the resulting factorization of~$a n$ only the archimedean condition associated with the factorizaton~$n' = d'_1 d'_2$ plays a role. Summarizing, we have
\begin{gather*}
  \sum_{\substack{a n = d_1 d_2\\d_1, d_2 > 0\\d_1 \equiv \beta + \td\beta \,\pmod{a}\\d_2 \equiv \beta - \td\beta \,\pmod{a}}}
  \big( d_1 \delta_{d_1 < d_2} + d_2 \delta_{d_2 < d_1} \big)
\;=\;
  \sum_{\substack{a n_a n' = d_1 d_2\\d_1, d_2 > 0\\d_1 \equiv \beta + \td\beta \,\pmod{a}\\d_2 \equiv \beta - \td\beta \,\pmod{a}}}
  \big( d_1 \delta_{d'_1 < d'_2} + d_2 \delta_{d'_1 < d'_2} \big)
\tx{,}
\end{gather*}
where $d'_1 = \gcd(d_1,n')$ and $d'_2 = \gcd(d_2,n')$. Since~$\gcd(a n_a, n') = 1$, we can sum over two factorizations~$a n_a = d_{a,1} d_{a,2}$ and $n' = d'_1 d'_2$. Every factor of~$n'$ is congruent to $1$~modulo~$a$, so the congruence condition applies only to the factors of~$a n_a$. We have
{\allowdisplaybreaks
\begin{align}
\nonumber
&\hphantom{{}\;=\;{}}
  \sum_{\substack{a n_a n' = d_1 d_2\\d_1, d_2 > 0\\d_1 \equiv \beta + \td\beta \,\pmod{a}\\d_2 \equiv \beta - \td\beta \,\pmod{a}}}
  \big( d_1 \delta_{\gcd(d_1,n') < \gcd(d_2,n')} + d_2 \delta_{\gcd(d_1,n') < \gcd(d_2,n')} \big)
\\
\nonumber
&\;=\;
  \sum_{\substack{a n_a = d_{a,1} d_{a,2}\\n' = d'_1 d'_2\\d_{a,1}, d_{a,2}, d'_1, d'_2 > 0\\d_{a,1} \equiv \beta + \td\beta \,\pmod{a}\\d_{a,2} \equiv \beta - \td\beta \,\pmod{a}}}
  \big( d_{a,1} d'_1 \delta_{d'_1 < d'_2} + d_{a,2} d'_2 \delta_{d'_2 < d'_1} \big)
\\
\label{eq:theta-theta-holomorphic-projection-coefficient-uncoupled-nonachemedian-achemedian}
&\;=\;
  \Big(
  \sum_{\substack{a n_a = d_1 d_2\\d_1, d_2 > 0\\d_1 \equiv \beta + \td\beta \,\pmod{a}\\d_2 \equiv \beta - \td\beta \,\pmod{a}}}
  d_1 + d_2
  \Big)
  \,\cdot\,
  \Big(
  \sum_{\substack{n' = d_1 d_2\\d_1, d_2 > 0}}
  d_1 \delta_{d_1 < d_2} + d_2 \delta_{d_2 < d_1}
  \Big)
\tx{.}
\end{align}}
By Lemma \ref{la:nonarchimedian-contribution}, the first factor in~\eqref{eq:theta-theta-holomorphic-projection-coefficient-uncoupled-nonachemedian-achemedian} is congruent to~$2$ modulo~$\ell$. 

We next inspect the second factor in~\eqref{eq:theta-theta-holomorphic-projection-coefficient-uncoupled-nonachemedian-achemedian}. We have
\begin{gather}
\label{eq:theta-theta-holomorphic-projection-coefficient-archimedean-contribution}
\begin{aligned}
  \sum_{\substack{q p = d_1 d_2\\d_1, d_2 > 0}}
  \big( d_1 \delta_{d_1 < d_2} + d_2 \delta_{d_2 < d_1} \big)
&\;=\;
  (1 + 1)
  (1 + q)
=
  2(1+q)
\\
  \sum_{\substack{q p^2 = d_1 d_2\\d_1, d_2 > 0}}
  \big( d_1 \delta_{d_1 < d_2} + d_2 \delta_{d_2 < d_1} \big)
&\;=\;
  (1 + 1)
  (1 + q)
  +
  (p + p)
=
  2(1 + q + p)
\tx{.}
\end{aligned}
\end{gather}

We now combine the archimedean and nonarchimedean factors in~\eqref{eq:theta-theta-holomorphic-projection-coefficient-uncoupled-nonachemedian-achemedian} to determine the Fourier coefficients of~\eqref{eq:eisenstein-theta-holomorphic-projection}. Our final expression in~\eqref{eq:eisenstein-theta-holomorphic-projection-coefficient-final} receives several contributions:
$1 \slash 16 \pi$ from~\eqref{eq:eisenstein-theta-holomorphic-projection};
$-4 \pi$ from~\eqref{eq:eisenstein-theta-holomorphic-projection-coefficient};
$2$ from~\eqref{eq:theta-theta-holomorphic-projection-coefficient-d12-factorizaton};
$2$ from Lemma~\ref{la:nonarchimedian-contribution}, computing the first factor in~\eqref{eq:theta-theta-holomorphic-projection-coefficient-uncoupled-nonachemedian-achemedian};
and~$2(1+p)$ or~$2(1 + q + p)$ from~\eqref{eq:theta-theta-holomorphic-projection-coefficient-archimedean-contribution}, computing the second factor in~\eqref{eq:theta-theta-holomorphic-projection-coefficient-uncoupled-nonachemedian-achemedian}.
This yields 
\begin{gather}
\label{eq:eisenstein-theta-holomorphic-projection-coefficient-final}
\begin{alignedat}{3}
  c\big(
  \pi^\hol_2\big( \rmU_{a,b}\, E_{\frac{3}{2}} \cdot (\theta_{a,\beta} + \theta_{a,-\beta}) \big);\,
&
  n_a q p
  \big)
&&\;\equiv\;
  -2 ( 1 + q)
&&
  \quad\pmod{\ell}
\tx{,}
\\
  c\big(
  \pi^\hol_2\big( \rmU_{a,b}\, E_{\frac{3}{2}} \cdot (\theta_{a,\beta} + \theta_{a,-\beta}) \big);\,
&
  n_a q p^2
  \big)
&&\;\equiv\;
  -2 ( 1 + q + p)
&&
  \quad\pmod{\ell}
\tx{.}
\end{alignedat}
\end{gather}
\end{proof}

We are now in a position to apply Theorem~\ref{thm:serre} and its Corollary~\ref{cor:serre}.
\begin{proof}[{Proof of Theorem~\ref{mainthm:hurwitz-congruences}}]

We are now in position to apply Theorem~\ref{thm:serre} to deduce a contradiction. Replacing $a$ and $b$ with $a a'$ and $b - a b'$, where $a'$ and $b'$ are as in Lemma~\ref{la:subprogression}, we may assume that $a$ and $b$ satisfy the conditions of Proposition \ref{prop:coefficientformula}.

Theorem~\ref{thm:serre} asserts that there are infinitely many primes~$p \equiv 1 \,\pmod{a \ell}$ such that $c(f;\,np^r) \equiv (r + 1) c(f;\,n) \,\pmod{\ell}$ for all $n \in \ZZ$ co-prime to~$p$, $r \in \ZZ_{\ge 0}$, and \linebreak $f \in \rmM_2(\Ga_1(4a))$. For sufficiently large $p$ there is a prime~$q$ with~$q \not\equiv 1 \,\pmod{\ell}$ and $q \equiv 1 \,\pmod{a}$ that satisfies~\eqref{eq:eisenstein-theta-holomorphic-projection-coefficient-prime-condition} (see Remark~\ref{rm:prop:coefficientformula}).

If the $q$\thdash\ Fourier coefficient of~$\pi^\hol_2\big( \rmU_{a,b}\, E_{\frac{3}{2}} \cdot (\theta_{a,\beta} + \theta_{a,-\beta}) \big)$ is divisible by~$\ell$, then the two congruences in~\eqref{eq:eisenstein-theta-holomorphic-projection-coefficient-final} yield the contradiction~$1 + q \equiv 1 + q + p \equiv 0 \,\pmod{\ell}$. Otherwise, they incur the relation
\begin{gather*}
  3 (1 + q)
\equiv
  2 (1 + q + p)
  \;\pmod{\ell}
\tx{,}
\end{gather*}
which is equivalent to~$q \equiv 1 \,\pmod{\ell}$, a contradiction. 
\end{proof}

\begin{appendix}

\section{Hecke-type congruences}
\label{sec:class-number-formula}

The fact that~$\sum_D H(D) e(D \tau)$ is a Hecke eigenform is conveniently captured by the Hurwitz class number formula. For fundamental discriminants~$-D$ and positive integers~$f$, the formulas for class numbers of imaginary quadratic fields and~$H(D)$ (see for example, pages~228 and~230 of~\cite{cohen-1993}) imply the following:
\begin{gather}
\label{eq:class-number-formula}
  H(D f^2)
\;=\;
  H(D)
  \frac{w(-D f^2)}{w(-D)}\,
  \sum_{d \isdiv f}
  d
  \prod_{p \isdiv d}
  \Big( 1 - \mfrac{1}{p}\left( \mfrac{-D}{p} \right) \Big)
\tx{,}
\end{gather}
where the product is over primes~$p$ dividing~$d$ and where~$w(-D) \isdiv 6$ is the number of roots of unity in the quadratic order of discriminant~$-D$. Throughout the paper we follow Zagier~\cite{zagier-1975}, who defines~$H(D)$ for nonnegative arguments, while Cohen~\cite{cohen-1993} uses the opposite sign convention. We restrict to congruences modulo powers of primes~$\ell > 3$ and hence may ignore the factor~$w(-D f^2) \slash w(-D)$. We assume no further knowledge of~$H(D)$, and we only employ the sum over divisors~$d$ of~$f$ in \eqref{eq:class-number-formula} to obtain congruences for~$H(a n + b)$. Note that this sum is multiplicative in~$f$, which later allows us to restrict to the case of prime powers~$a$.

From the introduction recall the congruences 
\begin{gather*}
  H(5^3 n + 25) \equiv 0 \;\pmod{5}
\tx{,}\quad
  H(7^3 n + 147) \equiv 0 \;\pmod{7}
\tx{,}\quad
  H(11^3 n + 242) \equiv 0 \;\pmod{11}
\text{.}
\end{gather*}
Observe that for $5^3 n + 25 = (5 n + 1) 5^2 = D f^2$ with a fundamental discriminant~$-D$ we have $D \equiv 1\, \pmod{5}$ and~$5 \isdiv f$. Thus, writing $f_5 \isdiv f$ for the highest power of~$5$ that divides $f$, we have $f_5 \ne 1$ and the right hand side of~\eqref{eq:class-number-formula} has the factor
\begin{gather*}
  \sum_{d \isdiv f_5}
  d 
  \prod_{5 \isdiv d}
  \Big( 1 - \mfrac{1}{5} \left( \mfrac{-D}{5} \right) \Big)
\equiv
  1
  \,-\,
  \left( \mfrac{-D}{5} \right)
\equiv
  0
  \;\pmod{5}
\tx{,}
\end{gather*}
which shows that $ H(5^3 n + 25) \equiv 0 \;\pmod{5}$.  The congruences modulo~$7$ and~$11$, and the congruences in the remark after Theorem~\ref{mainthm:hurwitz-congruences} follow similarly.

The above reasoning extends to powers~$\ell^m$ for arbitrary primes~$\ell$ as follows:  Assume that $a = \ell^e$ and $b = \ell^c u$ for an integer~$u$ with~$\gcd(\ell,u) = 1$. For simplicity, we further suppose that~$e > c \ge 2$. Set~$c' = \lfloor c \slash 2 \rfloor \ge 1$ and~$c'' = \min\{c',m\}$. We have the factorization
\begin{gather*}
  an + b
=
  \ell^e n + \ell^c u
=
  \ell^{2 c'}\, \ell^{c - 2c'}
  \big( \ell^{e - c} + u \big)
=
  f^2 D
\tx{,}
\end{gather*}
which yields~$\ell^{c'} \isdiv f$ and~$D \equiv 0 \,\pmod{\ell}$ if~$c$ is odd and~$D \equiv u \,\pmod{\ell}$ if~$c$ is even.
Thus, the right hand side of~\eqref{eq:class-number-formula} has the factor
\begin{multline*}
  \sum_{d \isdiv \ell^{c'}}
  d
  \prod_{\ell \isdiv d}
  \Big( 1 - \mfrac{1}{\ell} \left( \mfrac{-D}{\ell} \right) \Big)
=
  1
  \,+\,
  \sum_{n = 1}^{c'}
  \ell^n
  \Big( 1 - \mfrac{1}{\ell} \left( \mfrac{-D}{\ell} \right) \Big)
\equiv
  1
  +
  \sum_{n = 1}^{c''}
  \ell^{n-1}
  \Big( \ell - \left( \mfrac{-D}{\ell} \right) \Big)
\\
=
  \sigma_1\big( \ell^{c''-1}  \big)
  \Big( 1 - \left( \mfrac{-D}{\ell} \right) \Big)
  +
  \ell^{c''}
  \;\pmod{\ell^m}
\tx{.}
\end{multline*}
As a result, we find congruences modulo~$\ell^m$ if~$c \ge 2 m$ (and hence~$c'' \ge m$) is even and $-u$ is a square modulo~$\ell$. In particular, we obtain non-holomorphic Ramanujan-type congruences for all primes~$\ell > 3$.
\end{appendix}

\ifbool{nobiblatex}{%
  \bibliographystyle{alpha}%
  \bibliography{bibliography.bib}%
}{%
  \renewcommand{\baselinestretch}{.8}
  \Needspace*{4em}
  \begin{multicols}{2}
  \printbibliography[heading=none]%
  \end{multicols}
}

\Needspace*{3\baselineskip}
\noindent
\rule{\textwidth}{0.15em}

{\noindent\small
Olivia Beckwith\\
Department of Mathematics,
University of Illinois at Urbana-Champaign,
Urbana, IL 61801, USA\\
E-mail: \url{obeckwith@gmail.com}\\
Homepage: \url{https://sites.google.com/view/olivia-beckwith-homepage/home}
}\vspace{.5\baselineskip}

{\noindent\small
Martin Raum\\
Chalmers tekniska högskola och G\"oteborgs Universitet,
Institutionen för Matematiska vetenskaper,
SE-412 96 Göteborg, Sweden\\
E-mail: \url{martin@raum-brothers.eu}\\
Homepage: \url{http://raum-brothers.eu/martin}
}\vspace{.5\baselineskip}

{\noindent\small
Olav K. Richter\\
Department of Mathematics,
University of North Texas,
Denton, TX 76203, USA\\
E-mail: \url{richter@unt.edu}\\
Homepage: \url{http://www.math.unt.edu/~richter/}
}

\end{document}

%% file: packages.tex
\usepackage{etoolbox}
\usepackage{ifdraft}

\usepackage[utf8]{inputenc}
\usepackage[T1]{fontenc}

\usepackage{lmodern}
\usepackage{fourier}
\usepackage{amssymb, amsfonts}

\usepackage{mathrsfs} %
\usepackage{dsfont}
\usepackage[usenames,dvipsnames]{xcolor}%
\usepackage{lettrine}
\usepackage{url}

\usepackage{stmaryrd}

\usepackage[english]{babel}

\usepackage[draft=false]{scrlayer-scrpage}

\usepackage{wrapfig}
\usepackage{footmisc}
\usepackage{needspace}

\usepackage{amsmath}
\usepackage{nccmath}
\usepackage[thmmarks, amsmath, amsthm]{ntheorem}

\providebool{nobiblatex}
\boolfalse{nobiblatex}
\ifbool{nobiblatex}{%
}{%
  \usepackage[style=numeric-comp,giveninits,url=false,sortcites=true,maxbibnames=99]{biblatex}
  \bibliography{bibliography.bib}
}

\usepackage[pdftex,final,
            pdfborder={0 0 0}, colorlinks=true,
            linkcolor=BrickRed, citecolor=ForestGreen, urlcolor=RoyalBlue]{hyperref}

\ifdraft{%
\usepackage[colorinlistoftodos]{todonotes}
\usepackage{lineno}
}{}

\usepackage{booktabs}
\usepackage[inline]{enumitem}
\usepackage{float}
\usepackage{graphicx}
\usepackage{multicol}
\usepackage{scrtime}

%% file: layout.tex
\ifdraft{\date{\today}}{\date{}}

\ifdraft{%
\linenumbers
\newcommand*\patchAmsMathEnvironmentForLineno[1]{%
  \expandafter\let\csname old#1\expandafter\endcsname\csname #1\endcsname
  \expandafter\let\csname oldend#1\expandafter\endcsname\csname end#1\endcsname
  \renewenvironment{#1}%
     {\linenomath\csname old#1\endcsname}%
     {\csname oldend#1\endcsname\endlinenomath}}%
\newcommand*\patchBothAmsMathEnvironmentsForLineno[1]{%
  \patchAmsMathEnvironmentForLineno{#1}%
  \patchAmsMathEnvironmentForLineno{#1*}}%
\AtBeginDocument{%
\patchBothAmsMathEnvironmentsForLineno{equation}%
\patchBothAmsMathEnvironmentsForLineno{align}%
\patchBothAmsMathEnvironmentsForLineno{flalign}%
\patchBothAmsMathEnvironmentsForLineno{alignat}%
\patchBothAmsMathEnvironmentsForLineno{gather}%
\patchBothAmsMathEnvironmentsForLineno{multline}%
}}

\KOMAoptions{DIV=10}

\clearscrheadings
\automark[section]{subsection}

\renewcommand{\subsectionmark}[1]{}

\lohead{{\small \headertitle}}
\rohead{{\small \headerauthors}}

\RedeclareSectionCommand[%
  font=\Large\sffamily\bfseries,%
  beforeskip=1\baselineskip,%
  afterskip=0.5\baselineskip,%
  indent=0em%
  ]{section}

\RedeclareSectionCommands[%
  font=\normalfont\bfseries,%
  afterskip=-1em%
  ]{subsection,subsubsection}

\RedeclareSectionCommands[%
  font=\normalfont\itshape,%
  afterskip=-1em,%
  indent=0pt,
  ]{paragraph}

\ifbool{nobiblatex}{}{%
  \AtEveryBibitem{\clearfield{doi}}
  \AtEveryBibitem{\clearfield{isbn}}
  \AtEveryBibitem{\clearfield{issn}}
  \AtEveryBibitem{\clearfield{pages}}
  \AtEveryBibitem{\clearlist{language}}
}

\newenvironment{enumeratearabic}{
\begin{enumerate}[label=(\arabic*), leftmargin=0pt,labelindent=2em,itemindent=!]
}{
\end{enumerate}
}

\newenvironment{enumerateroman}{
\begin{enumerate}[label=(\roman*), leftmargin=0pt,labelindent=2em,itemindent=!]
}{
\end{enumerate}
}

\newenvironment{enumeratearabic*}{
\begin{enumerate*}[label=(\arabic*)] %
}{
\end{enumerate*}
}

\newenvironment{enumerateroman*}{
\begin{enumerate*}[label=(\roman*)] %
}{
\end{enumerate*}
}

\numberwithin{equation}{section}

\theoremnumbering{arabic}
\newtheorem{theoremcounter}{theoremcounter}[section]
\theoremnumbering{Alph}
\newtheorem{maintheoremcounter}{maintheoremcounter}

\theoremstyle{plain}

\newtheorem{corollary}[theoremcounter]{Corollary}
\newtheorem{lemma}[theoremcounter]{Lemma}

\newtheorem{proposition}[theoremcounter]{Proposition}
\newtheorem{theorem}[theoremcounter]{Theorem}

\theoremstyle{plain}

\newtheorem{maintheorem}[maintheoremcounter]{Theorem}

\theoremstyle{definition}

\theoremstyle{remark}

\newtheorem{remark}[theoremcounter]{Remark}

\theoremstyle{nonumberremark}

\newtheorem{mainremark}{Remark}
\newtheorem{remarkcomputation}{Computation}

\newenvironment{mainremarkenumerate}
{%
\mainremark
\enumeratearabic
}{%
\endenumeratearabic
\endmainremark
}%

%% file: shortcuts.tex
\ifdraft{
 
}{
 
}

\newcommand{\tx}{\ensuremath{\text}}

\newcommand{\tbf}{\bfseries}

\newcommand{\thdash}{\nbd th}

\newcommand{\nbd}{\nobreakdash-\hspace{0pt}}

\newcommand{\bboard}{\ensuremath{\mathbb}}
\newcommand{\cal}{\ensuremath{\mathcal}}

\newcommand{\bbM}{\ensuremath{\bboard M}}

\newcommand{\cO}{\ensuremath{\cal{O}}}

\newcommand{\rmd}{\ensuremath{\mathrm{d}}}

\newcommand{\rmt}{\ensuremath{\mathrm{t}}}

\newcommand{\rmF}{\ensuremath{\mathrm{F}}}

\newcommand{\rmM}{\ensuremath{\mathrm{M}}}

\newcommand{\rmU}{\ensuremath{\mathrm{U}}}

\newcommand{\td}{\tilde}

\newcommand{\ov}{\overline}

\newcommand{\ra}{\ensuremath{\rightarrow}}

\newcommand{\mto}{\ensuremath{\mapsto}}

\newcommand{\amid}{\ensuremath{\mathop{\mid}}}

\newcommand{\ZZ}{\ensuremath{\mathbb{Z}}}
\newcommand{\QQ}{\ensuremath{\mathbb{Q}}}

\newcommand{\CC}{\ensuremath{\mathbb{C}}}

\renewcommand{\Re}{\ensuremath{\mathrm{Re}}}
\renewcommand{\Im}{\ensuremath{\mathrm{Im}}}

\newcommand{\isdiv}{\amid}
\newcommand{\nisdiv}{\ensuremath{\mathop{\nmid}}}

\renewcommand{\pmod}[1]{\ensuremath{\;(\mathrm{mod}\, #1)}}

\newcommand{\SL}[1]{\ensuremath{\mathrm{SL}_{#1}}}

\newcommand{\U}[1]{\ensuremath{\mathrm{U}_{#1}}}

\newcommand{\T}{\ensuremath{\rmt}}

\newcommand{\HS}{\mathbb{H}}

%% file: shortcuts_this.tex
\newcommand{\ga}{\gamma}
\newcommand{\Ga}{\Gamma}
\newcommand{\hol}{\mathrm{hol}}
\newcommand{\ordp}{\operatorname{ord}_p}